\pdfoutput=1
\RequirePackage{ifpdf}
\ifpdf 
\documentclass[pdftex]{sigma}
\else
\documentclass{sigma}
\fi

\numberwithin{equation}{section}

\newtheorem{Theorem}{Theorem}[section]
\newtheorem*{Theorem*}{Theorem}

\newtheorem{Lemma}[Theorem]{Lemma}
\newtheorem{Proposition}[Theorem]{Proposition}
\newtheorem{question}[Theorem]{Question}
\newtheorem{Conjecture}[Theorem]{Conjecture}
 { \theoremstyle{definition}
\newtheorem{Definition}[Theorem]{Definition}
\newtheorem{Example}[Theorem]{Example}
\newtheorem{Remark}[Theorem]{Remark}
\newtheorem{MyParagraph}{}[subsection]
}

\usepackage{braket}
\usepackage{mathtools,stmaryrd}

\SetSymbolFont{stmry}{bold}{U}{stmry}{m}{n}
\usepackage{bm,bbm,mathrsfs}

\newcommand{\includepic}[1]{\vcenter{\hbox{\includegraphics[scale=1.1]{#1}}}}

\renewcommand{\hat}[1]{\widehat{#1}}
\newcommand{\hhat}[1]{\hat{\vphantom{\rule{1pt}{5.5pt}}\smash{\hat{#1}}}}
\renewcommand{\vec}[1]{\bm{#1}}
\renewcommand{\tilde}[1]{\widetilde{#1}}
\newcommand{\cat}[1]{\mathsf{#1}}
\newcommand{\lie}[1]{\mathfrak{#1}}
\newcommand{\bC}{\mathbb{C}}
\newcommand{\bP}{\mathbb{P}}
\newcommand{\bQ}{\mathbb{Q}}
\newcommand{\bZ}{\mathbb{Z}}
\newcommand{\bk}{\mathbbm{k}}
\newcommand{\cE}{\mathcal{E}}
\newcommand{\cF}{\mathcal{F}}
\newcommand{\cK}{\mathcal{K}}
\newcommand{\cM}{\mathcal{M}}
\newcommand{\cN}{\mathcal{N}}
\newcommand{\cO}{\mathcal{O}}
\newcommand{\cT}{\mathcal{T}}

\newcommand{\fq}{\mathfrak{q}}

\newcommand{\kQ}{Q}
\newcommand{\sA}{\mathsf{A}}
\newcommand{\sD}{\mathsf{D}}
\newcommand{\sF}{\mathsf{F}}
\newcommand{\sM}{\mathsf{M}}
\newcommand{\sT}{\mathsf{T}}
\newcommand{\sZ}{\mathsf{Z}}
\newcommand{\scA}{\mathscr{A}}
\newcommand{\sR}{\mathsf{R}}
\newcommand{\sS}{\mathsf{S}}
\newcommand{\sRR}{\mathsf{RR}}
\newcommand{\di}{\partial}

\newcommand{\pt}{\mathrm{pt}}
\newcommand{\vir}{\mathrm{vir}}
\newcommand{\loc}{\mathrm{loc}}
\newcommand{\chiq}{\chi_{\mathrm{q}}}
\newcommand{\chiqq}{\chi_{\mathrm{qq}}}
\DeclareMathOperator{\GL}{GL}
\DeclareMathOperator{\SU}{SU}
\DeclareMathOperator{\sign}{sign}
\DeclareMathOperator{\Hilb}{Hilb}
\DeclareMathOperator{\Taut}{\mathcal{T}\mkern-1.5mu{\mathit{aut}}}

\DeclareMathOperator{\End}{End}

\DeclareMathOperator{\tr}{tr}
\DeclareMathOperator{\PT}{PT}
\DeclareMathOperator{\Chow}{Chow}
\DeclareMathOperator{\Hom}{Hom}

\DeclareMathOperator{\Aut}{Aut}
\DeclareMathOperator{\Sym}{Sym}
\DeclareMathOperator{\Ext}{Ext}
\DeclareMathOperator{\Stab}{Stab}
\DeclareMathOperator{\Spec}{Spec}
\DeclarePairedDelimiter{\inner}{\langle}{\rangle}

\begin{document}
\allowdisplaybreaks

\newcommand{\arXivNumber}{2203.07072}

\renewcommand{\PaperNumber}{090}

\FirstPageHeading

\ShortArticleName{A Representation-Theoretic Approach to $qq$-Characters}

\ArticleName{A Representation-Theoretic Approach\\ to $\boldsymbol{qq}$-Characters}

\Author{Henry LIU}

\AuthorNameForHeading{H.~Liu}

\Address{Mathematical Institute, University of Oxford, Andrew Wiles Building,\\ Radcliffe Observatory Quarter, Woodstock Road, Oxford, OX26GG, UK}
\Email{\href{mailto:liu@maths.ox.ac.uk}{liu@maths.ox.ac.uk}}
\URLaddress{\url{http://people.maths.ox.ac.uk/liu/}}

\ArticleDates{Received May 15, 2022, in final form November 17, 2022; Published online November 24, 2022}

\Abstract{We raise the question of whether (a slightly generalized notion of) $qq$-characters can be constructed purely representation-theoretically. In the main example of the quantum toroidal $\mathfrak{gl}_1$ algebra, geometric engineering of adjoint matter produces an explicit vertex operator $\mathsf{RR}$ which computes certain $qq$-characters, namely Hirzebruch $\chi_y$-genera, completely analogously to how the R-matrix $\mathsf{R}$ computes $q$-characters. We give a geometric proof of the independence of preferred direction for the refined vertex in this and more general non-toric settings.}


\Classification{17B37; 17B67; 14N35}

\section{Introduction}

From a high vantage point, one could say this paper studies the
character theory of representations $V$ of quantum affine (or
affinized) algebras $U_q(\hat{\lie{g}})$. In the pioneering work
\cite{Frenkel1999}, from the R-matrix defining $U_q(\hat{\lie{g}})$,
Frenkel and Reshetikhin constructed the {\it $q$-character}~$\chiq(V)$, a~$q$-analogue of the ordinary character for
representations of classical Lie algebras. Later Nakajima~\cite{Nakajima2001} provided a geometric construction of $\chiq(V)$
using his quiver varieties. More recently, in~\cite{Nekrasov2016},
Nekrasov used an analogous geometric construction to produce a
one-parameter deformation of $\chiq(V)$ called the {\it
 $qq$-character} $\chiqq(V)$. Our main goal will be to bring this
collection of ideas full circle back to representation theory, and
attempt to fill in the remaining cell of the following table:
\[ \begin{array}{c|c|c}
 \text{Construction} & \chiq & \chiqq \\ \hline
 \text{Rep. theory} & \cite{Frenkel1999} & ?? \\
 \text{Geometry} & \cite{Nakajima2001} & \cite{Nekrasov2016}
 \end{array} \]
To this end, Section~\ref{sec:qq-characters} reviews the aforementioned
material, proposes a generalized notion of $qq$-character, and poses a
sequence of questions on how $\chiqq$ might be described purely in
terms of representation theory, namely using only the operators in
$U_q(\hat{\lie{g}})$ (as can be done for $q$-characters).
Schematically, we ask if there is an operator $\sRR_{W,V}$, analogous
to the R-matrix $\sR_{W,V}$, such that
\begin{gather} \label{eq:R-vs-RR-analogy}
 \sR_{W,V}\leadsto \chiq(V), \qquad
\sRR_{W,V}\leadsto \chiqq(V)
\end{gather}
are analogous procedures. We anticipate that such a
definition/description of $\chiqq$ will be useful, among other
applications, for studying the enumerative geometry of curves in
3-folds \cite{Liu2021}.

From a much lower, down-to-earth vantage point, this paper actually
mainly studies {\it geometric engineering} \cite{Iqbal2006, Katz1997},
a procedure by which certain K-theoretic Nekrasov partition functions
can be computed using networks of refined topological vertices
$C_{\lambda\mu\nu}(q,t)$. This is relevant because
\begin{itemize}\itemsep=0pt
\item a specific Nekrasov partition function $Z_r$ (for 5d $\cN=1^*$
 supersymmetric $\SU(r)$ Yang--Mills theory) is the Hirzebruch
 $\chi_y$-genus of the moduli of rank-$r$ instantons, which by our
 definition is a form of $qq$-character for (the $r$-fold tensor
 product of) the Fock module $\sF$ of the quantum toroidal algebra
 $U_{q,t}\big(\hhat{\lie{gl}}_1\big)$;
\item refined vertices have a history of descriptions \cite{Awata2012,
 Iqbal2010} using operators in $U_{q,t}\big(\hhat{\lie{gl}}_1\big)$.
\end{itemize}
Specifically, in Section~\ref{sec:geometric-engineering} we build from
refined vertices an operator $\sRR_{\sF,\sF}$ which produces $Z_r$ in
much the same way that the R-matrix $\sR_{\sF,\sF}$ produces
$\chiq(\sF^{\otimes r})$. In this way, we answer a question from
Section~\ref{sec:qq-characters} for the (important!) example of
$U_{q,t}\big(\hhat{\lie{gl}}_1\big)$. This result can be degenerated to~$U_q\big(\hat{\lie{sl}}_2\big)$.

The main challenge in finding an operator $\sRR_{W,V}$ is that
$\chiqq(V)$ is morally quadratic in~$V$, while $\chiq(V)$ is linear
(see Question~\ref{q:2}). For $U_{q,t}\big(\hhat{\lie{gl}}_1\big)$, the
desired quadratic terms naturally occur in (traces of) the so-called
{\it refined $4$-point function}, see, e.g.,~\cite{Awata2016}. This is our
$\sRR_{\sF,\sF}$. Compositions of these 4-point functions are similar
to but {\it distinct} from higher-rank Carlsson--Nekrasov--Okounkov
Ext operators~\cite{Carlsson2014}, which lack nice closed-form
formulas despite an excellent representation-theoretic
characterization \cite{Negut2017,Negut2020}. In contrast, in
Section~\ref{sec:explicit-operator-formula}, using a standard operator
formalism, we give an explicit formula for $\sRR_{\sF,\sF}$ involving
vertex operators with a~curious interaction of the horizontal and
vertical subalgebras of $U_{q,t}(\hhat{\lie{gl}}_1)$. Similar results
hold for other toroidal algebras if adjoint matter can be engineered.

Of significant independent interest is
Section~\ref{sec:dependence-on-preferred-direction}, where we prove the {\it
 independence of preferred direction}, also known as {\it slicing
 invariance}, of the network of refined vertices which computes
$Z_r$. This is a commonly-assumed property of appropriate networks
\cite{Awata2013, Iqbal2009, Morozov2016}. Recent work~\cite{Arbesfeld2021} of Arbesfeld contains a geometric proof if the
network is the toric skeleton of a smooth toric $3$-fold. We use a
degenerating family of abelian varieties to relax the toric
constraint, to allow for suitable non-toric gluings of edges. This
should cover all networks in current literature for which independence
of preferred direction is expected. The strategy in~\cite{Arbesfeld2021}, and for us as well, is to identify refined
vertices and partition functions as specific limits of equivariant
K-theoretic Pandharipande--Thomas (PT) vertices~\cite{Pandharipande2009} and partition functions. We believe this
geometric approach to be cleaner than the recent purely algebraic
proof of~\cite{Fukuda2020}, despite a dependence on the conjectural
K-theoretic DT/PT correspondence.

In accordance with
the analogy \eqref{eq:R-vs-RR-analogy}, in
Section~\ref{sec:properties-of-RR} we consider $\sRR^{\PT}$, the lift of
$\sRR$ to PT theory, and collect some (conjectural) properties of
$\sRR$ and $\sRR^{\PT}$ which we propose are analogues of certain
properties of the R-matrix $\sR$.

\section[qq-characters]{$\boldsymbol{qq}$-characters}\label{sec:qq-characters}

\subsection{The geometric definition}

\begin{MyParagraph}\label{sec:geometric-realization-modules}
Let $\Gamma$ be a quiver with vertices indexed by a set $I$.
Associated to a doubled and framed version of $\Gamma$ is the Nakajima
quiver variety
\[ X_\Gamma(\vec w) = \bigsqcup_{\vec v} X_\Gamma(\vec v, \vec w), \]
where $\vec v$ and $\vec w$ are dimension vectors of the ordinary and
framing vertices respectively in the quiver representation. Recall
that $X_\Gamma$ is a smooth algebraic symplectic variety. Let
$\sT = \sA \times \bC_{\hbar}^\times \subset \Aut(X_\Gamma)$ be a
(possibly maximal) torus such that $\bC_{\hbar}^\times$ scales the
symplectic form with weight $\hbar$ and $\sA$ preserves the symplectic
form.

The important work \cite{Nakajima2001}, and later \cite{Maulik2019,
 Okounkov2016}, showed that the equivariant K-theory group (of
coherent algebraic sheaves)
\begin{equation} \label{eq:equivariant-k-theory-quiver-variety}
 V_\Gamma(\vec w) \coloneqq K_\sT(X_\Gamma(\vec w)) 
\end{equation}
is a highest-weight module for a quantum group $\scA_\Gamma$ which is
essentially a quantum affinized algebra. For example, when $\Gamma$ is
of finite ADE type, $\scA_\Gamma = U_q\big(\hat{\lie{g}}_\Gamma\big)$ is the
quantum affine algebra for (a mild central extension of) the classical
Lie algebra $\lie{g}_\Gamma$. More relevant for us, if $\Gamma$ is the
Jordan quiver, with one vertex and one edge loop, then
$\scA_\Gamma = U_{q,t}\big(\hhat{\lie{gl}}_1\big)$ is the quantum toroidal~$\lie{gl}_1$ algebra, which is morally (but not literally) the quantum
affinization of $\hat{\lie{gl}}_1$.

The geometric realization
\eqref{eq:equivariant-k-theory-quiver-variety} is useful for studying
the representation theory of quantum groups, e.g., if $\Gamma$ of finite ADE
type then modules of the form $V_\Gamma(\vec w)$ form a basis in the
Grothendieck ring of all finite-dimensional $\scA_\Gamma$-modules.
\end{MyParagraph}

\begin{MyParagraph}
Associated to each module $V_\Gamma(\vec w)$ is the {\it $qq$-character}
$\chiqq(V_\Gamma(\vec w))$, originally introduced in \cite{Nekrasov2016} to
study the BPS/CFT correspondence.

\begin{Definition}
 Let $\Taut$ denote the tautological bundle of $X_\Gamma$, and let $f(-)$
 be a function on $K_\sT(X_\Gamma)$ such that $f(\cE_1 + \cE_2) = f(\cE_1)
 f(\cE_2)$. Set
 \begin{equation} \label{eq:qq-character-geometric}
 \chiqq^{(f)} (V_\Gamma(\vec w); m, \kQ ) \coloneqq \sum_{\vec v} \kQ^{\vec v} \chi_\sT\big(X_\Gamma(\vec v, \vec w), \wedge_{-m}^\bullet \big(\cT^\vee \big) \otimes f(\Taut)\big).
 \end{equation}
 Here $m$ and $\kQ = (\kQ_i)_{i \in I}$ are formal variables, with
 $\kQ^{\vec v}$ being short for $\prod_{i \in I} \kQ_i^{v_i}$, and
 $\wedge_{-m}^\bullet(-) \coloneqq \sum_i (-m)^i \wedge^i(-)$ is an
 alternating sum of exterior powers, $\cT$ is the tangent bundle, and
 \[ \chi_\sT(X, \cF) \coloneqq \sum (-1)^i H^i(X, \cF) \in K_\sT(\pt)_{\loc} \]
 is the $\sT$-equivariant Euler characteristic of a coherent sheaf
 $\cF$. (When $X$ is non-compact but the fixed locus $X^{\sT}$ is,
 $\chi_\sT$ is defined via $\sT$-equivariant localization, hence the
 subscript $\loc$.)
\end{Definition}
\end{MyParagraph}

\begin{MyParagraph}
Note that \eqref{eq:qq-character-geometric} is {\it not} the original
(combinatorial!) characterization of $\chiqq$ from \cite[Section~6.1]{Nekrasov2016}, and instead we have used the geometric formula
from \cite[Sections~8.3 and 8.4]{Nekrasov2016}. The geometric formula arises
from integration over a certain {\it moduli of crossed instantons},
see \cite{Nekrasov2017} for an ADHM-style construction.
Algebro-geometrically, this moduli space admits a description and
virtual cycle in the style of Oh--Thomas~\cite{Oh2020}.\footnote{Private communication with N. Arbesfeld.}

Combinatorially, (the original) $qq$-characters may be constructed by
recursive expansion \cite{Feigin2021} in a similar fashion as for
$q$-characters \cite{Frenkel2001}. This is a great approach for
explicit computation, especially for $\scA_\Gamma$-modules which are not
geometrically realizable like in
\eqref{eq:equivariant-k-theory-quiver-variety}, but it is not the
direction we will take in this paper.
\end{MyParagraph}

\begin{MyParagraph}
We have allowed for an {\it arbitrary} multiplicative function $f$ in
\eqref{eq:qq-character-geometric}, while the original $qq$-characters
use a specific function $f_{\psi}$ (namely the product of all
$f_{\psi,i}$ from~\eqref{eq:psi-geometric}). We feel that
$qq$-characters with more general $f$ should be studied on equal
footing, especially in light of connections \cite[Section~4.2]{Liu2021}
between $\chiqq^{(f)}$ and quantities in the enumerative geometry of
curves in $3$-folds where $f$ corresponds exactly to a {\it descendent
 insertion}.

An example, which may be of independent interest, is when $f(-) = \cO$
is the trivial constant function, which we denote $f = 1$. In this
case~\eqref{eq:qq-character-geometric} is nothing but the
(equivariant) Hirzebruch $\chi_y$-genus
\[ \chiqq^{(1)}(V_\Gamma(\vec w); m, \kQ) = \chi_{\sT,m}(X_\Gamma(\vec w); \kQ) \coloneqq \sum_{\vec v} \kQ^{\vec v} \chi_{\sT}\big(X_\Gamma(\vec v, \vec w), \wedge_{-m}^\bullet\big(\cT^\vee\big)\big), \]
though our variable is called $m$ instead of~$y$.
\end{MyParagraph}

\subsection{The question(s)}

\begin{MyParagraph}
We will now pose a sequence of successively more precise questions
about the re\-pre\-sen\-ta\-tion-theoretic nature of $qq$-characters. Let
$V_\Gamma(\vec w)$ be a geometric representation of $\scA_\Gamma$ as
in Section~\ref{sec:geometric-realization-modules}.

\begin{question}\label{q:0}
 Can the quantity $\chiqq^{(f)}(V_\Gamma(\vec w); m, \kQ)$ be expressed
 purely in terms of operators in $\scA_\Gamma$ acting on $V_\Gamma(\vec w)$?
\end{question}

This question is motivated by the following observation. The
specialization $m=1$ for $qq$-characters gives
\begin{align}
 \chiqq^{(f)}(V_\Gamma(\vec w); 1, \kQ)
 &= \chi_{\sT}\left(X_\Gamma(\vec w), \kQ^{\cdots} \cdot \wedge_{-1}^\bullet\left(\cT^\vee\right) \otimes f(\Taut)\right) \label{eq:q-character-geometric} \\
 &= \chi_{\sT}\left(X_\Gamma(\vec w) \times X_\Gamma(\vec w), \kQ^{\cdots} \cdot \iota_\Delta(f(\Taut))\right), \label{eq:q-character-as-trace}
\end{align}
where in \eqref{eq:q-character-geometric} we abbreviated $\sum_{\vec
 v} \kQ^{\vec v}$ as $\kQ^{\cdots}$, and in
\eqref{eq:q-character-as-trace} $\iota_\Delta$ is the inclusion of the
diagonal. Hence \eqref{eq:q-character-as-trace} is the trace, in
$V_\Gamma(\vec w)$, of the operator of multiplication by $f(\Taut)$. It is
known that such operators always live in a commutative subalgebra of
$\scA_\Gamma \subset \End(V_\Gamma(\vec w))$; see \cite[Section~5.4]{Maulik2019}
(written for cohomology/Yangians, but the general principle still
applies).
\end{MyParagraph}

\begin{MyParagraph}  For a specific choice of $f$, the $m=1$ specialization is exactly a $q$-character.
\begin{Example} \label{ex:q-character}
 Let $\Gamma$ be of finite ADE type, and let $\big\{\psi_i^\pm(u)\big\}_{i \in I}$
 be the Drinfeld generators of the loop Cartan in
 $\scA_\Gamma = U_q(\hat{\lie{g}}_\Gamma)$. Let $\Taut = \bigoplus_i \Taut_i$
 be the decomposition across vertices of~$\Gamma$, and consider
 \begin{equation} \label{eq:psi-geometric}
 f_{\psi,i}(\Taut) \coloneqq \hat S^\bullet\big(\big(1 - \hbar^{-1}\big) \otimes \Taut_i\big),
 \end{equation}
 where $\hat S^\bullet(V) \coloneqq \Sym^\bullet(V) \otimes (\det
 V)^{1/2}$ is a ``symmetrized'' version of the symmetric algebra.
 Then
 \begin{equation} \label{eq:psi-geometric-u}
 \psi_i^\pm(u) = f_{\psi,i}(u \otimes \Taut),
 \end{equation}
 as operators expanded in series around $u^{\pm 1} \to 0$, see, e.g.,
 \cite[Theorem~9.4.1]{Nakajima2001}. Therefore
 $\chiqq^{(f_{\psi,i})}(V_\Gamma(\vec w); 1, \kQ)$ equals
 \begin{equation} \label{eq:q-character}
 \chi_q^{(i)}(V_\Gamma(\vec w); \kQ) \coloneqq \tr_{V_\Gamma(\vec w)} Q^{\cdots} \psi_i^\pm(u),
 \end{equation}
 also known as the ($i$-th) {\it $q$-character} of $V_\Gamma(\vec w)$. Up
 to some syntactic repackaging, \eqref{eq:q-character} is essentially
 the $q$-character originally introduced in \cite{Frenkel1999} for
 quantum affine algebras. This explains the nomenclature
 ``$qq$-character''.
\end{Example}
\end{MyParagraph}

\begin{MyParagraph}
Whenever the canonical bundle $\cK_X$ admits a square root, there is a
natural pairing on $K_\sT(X)$ given by
\begin{equation} \label{eq:k-theory-hermitian-pairing}
 \inner{\cF_1, \cF_2}_X \coloneqq \sum_i (-1)^i \Ext^i_\sT\big(\cF_1, \cF_2 \otimes \cK_X^{1/2}\big),
\end{equation}
and
$\inner{\cF_1, \cF_2}_X = (-1)^{\dim X} \inner{\cF_2, \cF_1}_X^\vee$
if Serre duality applies. Since our $X$ will always be smooth and
symplectic, \eqref{eq:k-theory-hermitian-pairing} becomes a Hermitian
form. We therefore use bra-ket notation for elements of $K_\sT(X)$,
along with the shorthand
\[ \inner{v | A | w}' \coloneqq \frac{\inner{v | A | w}}{\inner{v | w}}. \]
\end{MyParagraph}

\begin{MyParagraph}
For our purposes, it is useful to repackage \eqref{eq:q-character} as
follows. Let $\vec\delta_i = (0, \ldots, 0, 1, 0, \ldots, 0)$ be the
dimension vector with $1$ in the $i$-th position only, and consider
the highest-weight modules $V_i \coloneqq V_\Gamma(\vec\delta_i)$. Let
$\ket{\varnothing}$ denote the highest weight vector (up to scalars).

\begin{Proposition}
 Let $\Gamma$ be of finite ADE type. Let $V$ be a finite-dimensional
 $\scA_\Gamma$-module and $\sR_{V_i,V} \in \End(V_i \otimes V)$ be the
 R-matrix. Then
 \begin{equation} \label{eq:R-matrix-vacuum-elements}
 \psi_i^\pm(u) = \inner{\varnothing | \sR_{V_i,V} | \varnothing}_1' \in \End(V),
 \end{equation}
 where $\inner{-}_1$ means to take the matrix element in the first
 tensor factor $V_i$.
\end{Proposition}

In fact, in general the entire Hopf algebra $\scA_\Gamma$, not just
its elements $\psi_i^\pm(u)$, is constructed from the data of
R-matrices $\{\sR_{V,V'}\}$ where $V$, $V'$ range over the modules~\eqref{eq:equivariant-k-theory-quiver-variety} (\cite{Reshetikhin1989},
\cite[Section~3]{Okounkov2016}  or \cite[Section~5]{Maulik2019} for the
cohomological case). When $\Gamma$ has loops, in general
$\psi_i^\pm(u)$ is a product of the matrix elements in the r.h.s.\ of~\eqref{eq:R-matrix-vacuum-elements}.

\begin{proof}[Proof sketch.]
 Recall $\sR = \Stab_-^{-1} \circ \Stab_+$ where $\Stab_\pm$ are
 upper-/lower-triangular stable envelopes. So only the diagonal
 terms, normalized exactly to give \eqref{eq:psi-geometric-u},
 contribute to $\inner{\varnothing | \sR | \varnothing}$; see the
 factorization of $\sR$ in \cite[Section~2.3]{Okounkov2016} or the
 cohomological argument in \cite[Section~4.7]{Maulik2019}.
\end{proof}

Therefore
\begin{equation} \label{eq:q-character-as-vacuum-matrix-element}
 \chiq^{(i)}(V; \kQ) = \tr_V \kQ^{\cdots} \inner{\varnothing | \sR_{V_i,V} | \varnothing}_1'.
\end{equation}
\end{MyParagraph}

\begin{MyParagraph}
In complete analogy with~\eqref{eq:q-character-as-vacuum-matrix-element}, a natural refinement
of Question~\ref{q:0} is the following.

\begin{question}\label{q:1}
 Does there exist a highest-weight module $W$ and an operator
 $\sRR_{W,V}(m)$ such that, for appropriate $f$,
 \begin{equation} \label{eq:abstract-RR-matrix-for-qq-character}
 \chiqq^{(f)}(V; m, \kQ) = \tr_V \kQ^{\cdots} \frac{\inner{\varnothing | \sRR_{W,V}(m) | \varnothing}_1'}{\inner{\varnothing \otimes \varnothing | \sRR_{W,V}(m) | \varnothing \otimes \varnothing}'},
 \end{equation}
 with $\sRR_{W,V}(1) = \sR_{W,V}$?
\end{question}

With our normalization conventions,
$\inner{\varnothing \otimes \varnothing | \sR_{W,V} | \varnothing \otimes \varnothing}' = 1$
and therefore this factor was not present in
\eqref{eq:q-character-as-vacuum-matrix-element}. For brevity, let
\[ \inner{\varnothing | \sRR | \varnothing}_1'' \]
denote the normalized operator in the r.h.s.\ of
\eqref{eq:abstract-RR-matrix-for-qq-character}, so
$\chiqq^{(f)} = \tr \kQ^{\cdots} \inner{\varnothing | \sRR | \varnothing}_1''$.
\end{MyParagraph}

\begin{MyParagraph}
We take a very specific and somewhat naive approach to answering
Question~\ref{q:1} for the modules $V = V_\Gamma(\vec w)$. Let
$X = X_\Gamma(\vec w)$, and let $\{\ket{\cO_p}\}_{p \in X^\sT}$ be the
basis of structure sheaves of fixed points. The operator of
multiplication by $\Taut$ acts diagonally in this basis.

\begin{question}\label{q:2}
 Does there exist a highest-weight module $W$ and an operator
 $\sRR_{W,V}(m)$ such that
 \[ \frac{\wedge_{-m}^\bullet(\cT_p^\vee)}{\wedge_{-1}^\bullet(\cT_p^\vee)} = \frac{\inner*{\varnothing \otimes \cO_p | \sRR_{W,V}(m) | \varnothing \otimes \cO_p}'}{\inner*{\varnothing \otimes \varnothing | \sRR_{W,V}(m) | \varnothing \otimes \varnothing}'} \]
 for every $p \in X^{\sT}$?
\end{question}

Such an operator would essentially answer Question~\ref{q:1} since, by
$\sT$-equivariant localization,
\begin{equation} \label{eq:qq-character-in-fixed-point-basis}
 \chiqq^{(f)}(V; m, \kQ) = \tr_V \kQ^{\cdots} f(\Taut) \inner*{\varnothing | \sRR_{W,V}(m) | \varnothing}_1''.
\end{equation}
\end{MyParagraph}

\begin{MyParagraph}
In Section~\ref{sec:geometric-engineering}, we answer the following variant
of Question~\ref{q:2} in the affirmative for the Jordan quiver $\Gamma$,
for which the Nakajima quiver variety $X_\Gamma(r) = \cM_r$ is the moduli
of rank-$r$ instantons.

\begin{question} \label{q:2p}
 Does there exist a highest-weight module $W$ and an operator
 $\sRR_{W,V}(m)$ such that
 \[ \frac{\wedge_{-m}^\bullet(\cT_p^\vee)}{\wedge_{-1}^\bullet(\cT_p^\vee)} = \frac{\inner*{\varnothing \otimes \cF_p | \sRR_{W,V}(m) | \varnothing \otimes \cF_p}'}{\inner*{\varnothing \otimes \varnothing | \sRR_{W,V}(m) | \varnothing \otimes \varnothing}'} \]
 for every $p \in X^{\sT}$, for some basis
 $\{\ket{\cF_p}\}_{p \in X^{\sT}}$ with nice properties?
\end{question}

One reason to look beyond the basis $\{\ket{\cO_p}\}$ of fixed points
is that fixed points do not behave nicely with respect to tensor
product; see Section~\ref{sec:instanton-moduli-bases} for details in the
case of $\cM_r$. For example, any operator $\sRR_{W,V}$ satisfying the
original Question~\ref{q:2} has little hope of satisfying the fusion
property of the original R-matrices $\sR_{W,V}$ (see
Section~\ref{sec:RR-matrix-fusion}), but our operator $\sRR_{W,V}$ will
satisfy Question~\ref{q:2p} for a basis $\{\ket{\cO^{\otimes}_p}\}$
preserving this fusion property.
\end{MyParagraph}

\begin{MyParagraph}
The discrepancy between Questions~\ref{q:2} and \ref{q:2p} is exactly
the obstruction to incorporating a~non-trivial insertion $f$. Namely,
we answer Question~\ref{q:2p} in a basis $\{\ket{\cO_p^\otimes}\}$
where $\Taut$ does {\it not} act diagonally (see
Section~\ref{sec:fixed-point-vs-naive-rank-r-basis}), and therefore the
resulting $\sRR_{W,V}(m)$ no longer answers Questions~\ref{q:1} or
\ref{q:2}; only the specific $f=1$ case
\[ \chiqq^{(1)}(V; m, \kQ) = \tr_V \kQ^{\cdots} \inner*{\varnothing | \sRR_{W,V}(m) | \varnothing}_1'' \]
of \eqref{eq:qq-character-in-fixed-point-basis} continues to hold. To
incorporate a general $f$, one could try to rewrite $\sRR(m)$ in the
true fixed-point basis $\{\ket{\cO_p}\}$. Put differently, if $\sS$ is
the change of basis from $\{\ket{\cO_p}\}$ to
$\{\ket{\cO_p^\otimes}\}$, then $\sS^{-1}\cdot \sRR\cdot \sS$
satisfies Question~\ref{q:2} if the operator $\sRR$ satisfies
Question~\ref{q:2p}. However, it is unclear whether there is a
representation-theoretic interpretation or formula for $\sS$.
\end{MyParagraph}

\begin{MyParagraph}\label{sec:RR-spectrum-vs-trace}
To be clear, the conditions imposed by Questions~\ref{q:2} and
~\ref{q:2p} are on the {\it spectrum} of the operator~$\sRR$, while
the condition of Question~\ref{q:1} is merely on the trace of $\sRR$.
One therefore expects the former to be far more stringent than the
latter. Indeed, we see this explicitly as a~consequence of the results
in Section~\ref{sec:geometric-engineering}, where we find many different
operators $\sRR$ such that
\[ \chi_{\sT,m}(\cM_r; \kQ) = \tr \kQ^{\cdots} \inner{\varnothing | \sRR | \varnothing}_1'' \]
is the Hirzebruch $\chi_y$-genus of $\cM_r$, but only one such $\sRR$
has the ``correct'' diagonal elements.
\end{MyParagraph}

\section{Geometric engineering}
\label{sec:geometric-engineering}

\subsection{The setup}
\label{sec:geometric-engineering-setup}

\begin{MyParagraph}
Let $\Gamma$ be the Jordan quiver, with one vertex and
one edge loop. The Nakajima quiver variety $X_\Gamma(r)$ is the moduli
\[ \cM_r \coloneqq \cM_r(\bC^2) \coloneqq \big\{E \in \cat{Coh}(\bP^2) \text{ torsion-free}\colon E\big|_{\bP^1_\infty} \cong \cO_{\bP^1_\infty}^{\oplus r}\big\} \]
of rank-$r$ instantons on $\bC^2$. It admits natural actions induced
by $\GL_2$ acting on $\bC^2$, and by $\GL_r$ acting on the framing
$\cO_{\bP^1_\infty}^{\oplus r}$. Let
\[ \sT \coloneqq \sT_{\text{framing}} \times \sT_{\bC^2} \coloneqq (\bC^\times)^r \times (\bC^\times)^2 \ni (a_1, \ldots, a_r, q, t) \]
be the maximal torus of this $\GL_r \times \GL_2$, with coordinates
written as above. The $r = 1$ case is the Hilbert scheme $\cM_1 =
\Hilb$ of points on $\bC^2$.
\end{MyParagraph}

\begin{MyParagraph}
Let
$\bk \coloneqq K_{\sT}(\pt)_{\loc} = \bZ\big[a_1^\pm, \ldots, a_r^\pm, q^\pm, t^\pm\big]_{\loc}$
where $\loc$ means to adjoin $(1-w)^{-1}$ for all non-zero monomials
$w \in K_{\sT}(\pt)$. All our modules and computations are implicitly
over this base ring. Set
\[ \sF \coloneqq K_\sT(\Hilb). \]
The quantum group associated to $\Gamma$ is the quantum toroidal
$\lie{gl}_1$ algebra $\scA_\Gamma = U_{q,t}\big(\hhat{\lie{gl}}_1\big)$, and $\sF$
is its standard Fock module. See
Section~\ref{sec:quantum-toroidal-gl1-details} for some more details. By
general principles \cite{Maulik2019}, stable envelopes provide an
isomorphism of $\scA_\Gamma$-modules
\begin{equation} \label{eq:k-theory-instanton-moduli}
 \Stab\colon \ \sF^{\otimes r} \otimes \bk \xrightarrow{\sim} K_\sT(\cM_r) \otimes \bk.
\end{equation}
Stable envelopes can be viewed as ``corrected'' versions of (the
pushforward along) the inclusion
\[ \iota\colon \ (\cM_r)^{\sT_{\text{framing}}} = \Hilb^{\times r} \hookrightarrow \cM_r \]
of the $\sT_{\text{framing}}$-fixed locus, which itself only induces
an isomorphism of $\bk$-modules.
\end{MyParagraph}

\begin{MyParagraph}
The $\sT_{\bC^2}$-fixed points of $\Hilb$, and therefore basis
elements of $\sF$, are labeled by partitions~$\lambda$. The
$\sT$-fixed points of $\cM_r$ are therefore $r$-tuples $\vec\lambda =
\big(\lambda^{(1)}, \ldots, \lambda^{(r)}\big)$ of partitions. So here we fix
some notation for partitions.

We will use the letters $\lambda$, $\mu$, $\nu$ to denote
partitions. Occasionally it is useful to view a~partition
$\lambda = (\lambda_1, \lambda_2, \ldots)$ in terms of its Young
diagram. Let $(i, j) = (i(\square), j(\square))$ be the position of a~square $\square \in \lambda$ in the Young diagram. If $\lambda^t$ is
the conjugate partition, set
\[ a(\square) \coloneqq \lambda_{i(\square)} - j(\square), \qquad \ell(\square) \coloneqq \big(\lambda^t\big)_{j(\square)} - i(\square). \]
Let $|\lambda| \coloneqq \sum_i \lambda_i$ denote the size of~$\lambda$.
\end{MyParagraph}

\begin{MyParagraph}\label{sec:instanton-moduli-bases}
We identify $\sF$ with the algebra of symmetric functions. From
\eqref{eq:k-theory-hermitian-pairing}, there is a standard inner
product $\inner{-, -}$ on $\sF$. We will use two bases in $\sF$:
\begin{itemize}\itemsep=0pt
\item the basis of Schur polynomials $s_\lambda$, which are
 orthonormal;
\item the basis of fixed points $\cO_\lambda$, which are orthogonal
 with norm
 \[ \inner{\cO_\lambda, \cO_\lambda} = (qt)^{\frac{|\lambda|}{2}} \prod_{\square \in \lambda} \big(1 - q^{-\ell(\square)-1} t^{a(\square)}\big) \big(1 - q^{\ell(\square)} t^{-a(\square)-1}\big). \]
 This can also be taken as the definition of $\inner{-, -}$ if
 desired. Let $\tilde\cO_\lambda$ be the unit vector normalization of
 $\cO_\lambda$, so they form an ortho{\it normal} basis.
\end{itemize}
The elements $\cO_\lambda \in \sF$ are Haiman's normalization of
Macdonald polynomials, e.g., denoted $\tilde H_\lambda$ in~\cite{Haiman1999}.
\end{MyParagraph}

\begin{MyParagraph}\label{sec:fixed-point-vs-naive-rank-r-basis}
We identify $\sF^{\otimes r}$ with $K_\sT(\cM_r)$ using
\eqref{eq:k-theory-instanton-moduli}. The tensor product
$\sF^{\otimes r}$ inherits the inner product from $\sF$; note that
this is {\it not} the inner product on $K_\sT(\cM_r)$ from
\eqref{eq:k-theory-hermitian-pairing}. We consider two bases in
$\sF^{\otimes r}$:
\begin{itemize}\itemsep=0pt
\item the basis of generalized Schur polynomials
 $s_{\vec\lambda} \coloneqq \otimes_{i=1}^r s_{\lambda^{(i)}}$, which
 are orthonormal;
\item the basis of fixed points $\cO_{\vec\lambda}$, also known as
 {\it generalized} Macdonald polynomials (see, e.g.,~\cite{Awata2016},
 cf.~\cite{Smirnov2014}), which are not orthogonal.
\end{itemize}
The latter is not the same as
\[ \cO^{\otimes}_{\vec\lambda} \coloneqq \otimes_{i=1}^r \cO_{\lambda^{(i)}} \in \sF^{\otimes r}, \]
which do form an orthonormal basis. This is because although
$\cO_{\vec\lambda} = \iota_* \boxtimes_{i=1}^r \cO_{\lambda^{(i)}}$ as
sheaves on~$\cM_r$, the identification
\eqref{eq:k-theory-instanton-moduli} is not $\iota_*$. A crucial
distinguishing property is that, in general,
\[ \cO_{(\lambda^{(1)}, \lambda^{(2)})} \neq \cO_{\lambda^{(1)}} \otimes \cO_{\lambda^{(2)}}. \]
\end{MyParagraph}

\subsection[The chi\_y-genus]{The $\boldsymbol{\chi_y}$-genus}

\begin{MyParagraph}
Our primary goal is to study the $\chi_y$-genus
$\chi_{\sT,m}(\cM_r; \kQ)$, and various representation-theoretic
description of it in terms of operators in the algebra
$U_{q,t}\big(\hhat{\lie{gl}}_1\big)$. The strategy is to compute the
$\chi_y$-genus via a form of {\it geometric engineering}, which
equates it to the refined partition functions of certain special toric
$3$-folds. One consequence, among others, of our computations is an
affirmative answer (Theorem~\ref{thm:instanton-RR-operator}) to
Question~\ref{q:2p}.
\end{MyParagraph}

\begin{MyParagraph}
Refined partition functions in the toric setting are (sums of)
products of contribu\-tions~$C_{\lambda\mu\nu}(q,t)$, called {\it
 refined vertices}, from each toric chart. One labels each edge of
the toric $1$-skeleton with a partition and performs certain
combinatorial sums over them; see \cite{Iqbal2009} for details. For
example,
\[ \includepic{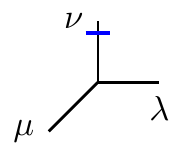} = C_{\lambda\mu\nu}(q,t), \qquad \includepic{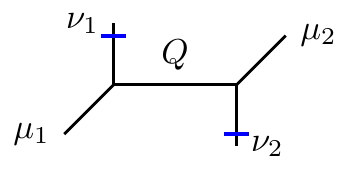} = \sum_\lambda Q^{|\lambda|} C_{\lambda\mu_1\nu_1}(q,t) C_{\lambda^t\mu_2\nu_2}(t,q). \]
Here a marked half-edge
\includegraphics{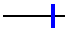} labels
the ``preferred direction'' $\nu$ of the vertex $C_{\lambda\mu\nu}$,
which is necessary because it is not symmetric in $\lambda$, $\mu$, $\nu$.
(This asymmetry is evident in the explicit formula~\eqref{eq:refined-vertex} later.) An unlabeled half-edge is set to
$\varnothing$, and any other edge not explicitly labeled by a partition
is summed over. Each edge may be labeled with a so-called {\it
 K\"ahler variable}, e.g.,~$Q$, indicating a term $Q^{|\cdot|}$
recording the size of the partition on the edge. The result is a~function of~$q$,~$t$, and various K\"ahler variables.

\begin{Remark}
 For the experts, all our edges will have normal bundles
 $\cO(-1)^{\oplus 2}$, i.e., everything is locally a conifold, so we
 neglect framing factors when gluing refined vertices.
\end{Remark}
\end{MyParagraph}

\begin{MyParagraph}
Different choices of toric diagram engineer different quantities on
$\cM_r$, or more generally $\cM_{r_1} \times \cdots \times \cM_{r_k}$,
and a general recipe is given in~\cite{Katz1997}. However, it is
important that the diagram does {\it not} need to be the $1$-skeleton
of an actual toric $3$-fold for geometric engineering to work. In
particular, the $\chi_y$-genus involves gluing edges in the following
non-toric way.

\begin{Proposition} \label{prop:engineering-adjoint-matter}
 With the substitution $t \mapsto t^{-1}$,
 \begin{equation} \label{eq:adjoint-matter-diagram}
 \chi_{\sT,\kappa}(\cM_r; Q) = \includepic{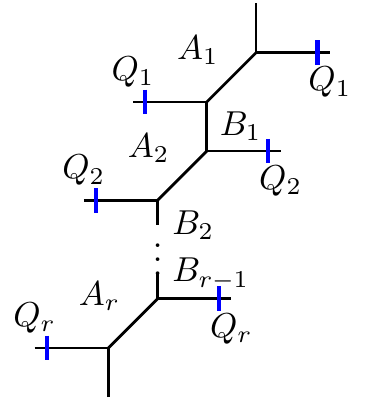} \bigg/ \includepic{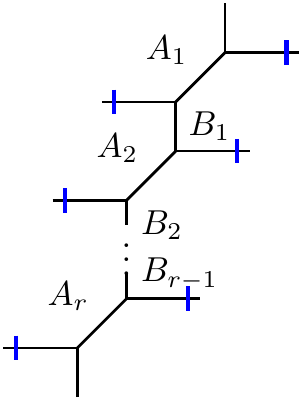}
 \end{equation}
 where the half-edges with the same variable $Q_i$ are glued
 together, i.e., there is an additional overall sum $\sum_{\lambda_1,
 \ldots, \lambda_r} Q_1^{|\lambda_1|} \cdots Q_r^{|\lambda_r|}$,
 and there are identifications
 \begin{equation} \label{eq:kahler-variable-identifications}
 Q_1 = \cdots = Q_r = Q, \qquad A_1 = \cdots = A_r = \kappa, \qquad A_k B_k = a_{k+1}/a_k.
 \end{equation}
\end{Proposition}

\begin{proof} By explicit calculation using the formalism in~\cite{Iqbal2006} (or otherwise).
\end{proof}

In physics language, Proposition~\ref{prop:engineering-adjoint-matter}
gives the ``toric'' diagram for engineering adjoint $U(r)$ matter. The
denominator is referred to as the {\it perturbative} term, and has an
explicit closed form formula unimportant to us.
\end{MyParagraph}

\begin{MyParagraph}
Let $Z_r\big(q,t^{-1}; Q, A, \vec B\big)$ be the quantity encoded by the
numerator of~\eqref{eq:adjoint-matter-diagram}; since all the~$Q_i$
are specialized to be equal, we retain only a single variable denoted
$Q$, and similarly for the~$A_i$ and~$A$. For clarity, $Z_r$ is
written out explicitly in
\eqref{eq:adjoint-diagram-formula-horizontal}.

Let
$\vec\lambda = \big(\lambda^{(1)}, \lambda^{(2)}, \ldots, \lambda^{(r)}\big)$
be the partitions labeling the horizontal legs of~\eqref{eq:adjoint-matter-diagram}, so that we can write the individual
contributions of the diagram with fixed horizontal legs as
\[ Z_r(q,t;Q,A,\vec B) \eqqcolon \sum_{\vec\lambda} Q^{|\vec\lambda|} Z_r(q,t; A,\vec B)_{\vec\lambda}. \]
For example, the denominator of~\eqref{eq:adjoint-matter-diagram} is
$Z_r(\cdots)_{\vec\varnothing}$.

\begin{Theorem} \label{thm:instanton-RR-operator}
 There is an operator $\sRR \in \End(\sF \otimes \sF)[[Q, A]]$, with
 explicit formula given by~\eqref{eq:instanton-four-point-horizontal-operator-formula}, such that
 \begin{gather}
 \big({-}\sqrt{qt}\big)^{|\vec\lambda|} Z_r\big(q, t^{-1}; A, A\vec B\big)_{\vec\lambda} \nonumber\\
 \qquad{}= \big\langle \varnothing \otimes \cO^{\otimes}_{\vec\lambda} \big| \sRR^{(10)} B_1^{|\cdot|} \sRR^{(20)} B_2^{|\cdot|} \cdots B_{r-1}^{|\cdot|} \sRR^{(r0)} \big| \varnothing \otimes \cO^{\otimes}_{\vec\lambda}\big\rangle,\label{eq:instanton-RR-operator-matrix-elements}
 \end{gather}
 where $\sRR^{(ij)}$ means to act on the $i$-th and $j$-th tensor
 factors and the $B_i^{|\cdot|}$ act in the $0$-th tensor factor.
\end{Theorem}
\end{MyParagraph}

\begin{MyParagraph}
In fact, the proof of
Proposition~\ref{prop:engineering-adjoint-matter} proceeds by identifying
\[ \frac{\wedge^\bullet_{-\kappa} \big(\cT_{\vec\lambda}^\vee\big)}{\wedge^\bullet_{-1} \big(\cT_{\vec\lambda}^\vee\big)} = \frac{Z_r(\cdots)_{\vec\lambda}}{Z_r(\cdots)_{\vec\varnothing}} \]
up to the identifications \eqref{eq:kahler-variable-identifications}.
Hence Theorem~\ref{thm:instanton-RR-operator} resolves
Question~\ref{q:2p} in the affirmative for~$\cM_r$. Note that the
operator
\[ \big\langle \varnothing \big| \sRR^{(10)} B_1^{|\cdot|} \sRR^{(20)} B_2^{|\cdot|} \cdots B_{r-1}^{|\cdot|} \sRR^{(r0)} \big| \varnothing\big\rangle \in \End\big(\sF^{\otimes r}\big) \]
is closely related to but is {\it not} exactly the higher-rank
Carlsson--Nekrasov--Okounkov Ext operator \cite{Carlsson2014}, for
which no explicit vertex operator formula is known. The diagonal
matrix elements match but off-diagonal ones differ by some explicit
factors. The higher-rank Ext operator is a well-studied object in part
due to its role in the AGT correspondence, see, e.g.,~\cite{Negut2018},
and it would be interesting to relate known characterizations of it to
our explicit operator.
\end{MyParagraph}

\begin{MyParagraph}
Taking the trace of \eqref{eq:instanton-RR-operator-matrix-elements}
gives
\begin{equation} \label{eq:instanton-RR-operator-trace}
 Z_r\big(q,t^{-1}; -Q\sqrt{qt}, A, A\vec B\big) = \tr_{\sF^{\otimes r}} Q^{|\cdot|} \big\langle\varnothing \big| \sRR^{(10)} B_1^{|\cdot|} \sRR^{(20)} B_2^{|\cdot|} \cdots B_{r-1}^{|\cdot|} \sRR^{(r0)} \big| \varnothing\big\rangle_0,
\end{equation}
where $\inner*{-, -}_0$ means the matrix element is taken in the
$0$-th tensor factor. We will construct different operators $\sRR^H$
and $\sRR^V$ which both satisfy \eqref{eq:instanton-RR-operator-trace}
up to mild changes of variables (nb. the discussion of
Section~\ref{sec:RR-spectrum-vs-trace}). The key idea
(Theorem~\ref{thm:adjoint-matter-independence-of-preferred-direction})
is that the diagrams in~\eqref{eq:adjoint-matter-diagram} remain the
same under {\it any} choice of preferred direction for the refined
vertices. In particular, $\sRR^H$ (resp.~$\sRR^V$) arises from
horizontal (resp. vertical) preferred direction. The effects, not
always trivial, of changing the preferred direction can be
investigated quite generally for toric diagrams possibly with some
pairs of parallel half-edges glued together (in a non-toric way), and
so Section~\ref{sec:dependence-on-preferred-direction} is of independent interest.
\end{MyParagraph}

\subsection{An explicit operator formula}
\label{sec:explicit-operator-formula}

\begin{MyParagraph}\label{sec:quantum-toroidal-gl1-details}
The algebra $U_{q,t}\big(\hhat{\lie{gl}}_1\big)$ is complicated; see
\cite{Schiffmann2012} for various presentations. For us, it suffices
to know that it contains two special subalgebras:
\begin{itemize}\itemsep=0pt
\item the ``horizontal'' Heisenberg subalgebra, with generators
 $\{\alpha_n\}_{n \in \bZ}$ which in terms of power-sum polynomials
 $p_k \in \sF$ act as
 \[ \alpha_n = n\frac{\di}{\di p_n}, \qquad \alpha_{-n} = p_n, \qquad \forall n > 0; \]
\item the ``vertical'' commutative subalgebra, with generators
 $\{H_n\}_{n \in \bZ}$ which act on fixed points as $H_n \cdot \cO_\lambda
 = h_n(\lambda;q,t) \cO_\lambda$ with eigenvalue
 \begin{equation} \label{eq:tautological-weight-hilb}
 h_n(\lambda;q,t) \coloneqq \sign(n) \left(-\chi_\lambda(q^n, t^n) + \frac{1}{(1 - q^n)(1 - t^n)}\right)
 \end{equation}
 where
 $\chi_{\lambda}(q,t) \coloneqq \sum_{\square \in \lambda} q^{i(\square)} t^{j(\square)}$.
 One recognizes this as the weight of the tautological bundle of
 $\Hilb$ at the fixed point $\lambda$.
\end{itemize}
Experts will notice that the action of $\alpha_{-n}$ for $n > 0$ is
scaled by a factor $-\big(q^{n/2}-q^{-n/2}\big)\big(t^{n/2}-t^{-n/2}\big)$ from the
usual horizontal generators.
\end{MyParagraph}

\begin{MyParagraph}
Let
\begin{gather}
 \Gamma_+(z) \coloneqq \exp\left(\sum_{n>0} (qt)^{\frac{n}{2}} \big(t^{\frac{n}{2}} - t^{-\frac{n}{2}}\big) (H_n \otimes \alpha_n) \frac{z^n}{n}\right),\nonumber \\
 \Gamma_-(z) \coloneqq \exp\left(-\sum_{n>0} (qt)^{-\frac{n}{2}} \big(q^{\frac{n}{2}} - q^{-\frac{n}{2}}\big) (H_{-n} \otimes \alpha_{-n}) \frac{z^{-n}}{n}\right).\label{eq:gamma-vertex-operators}
 \end{gather}
Furthermore let $\sD$ be the diagonal operator whose entries are
$\big\langle \tilde\cO_\lambda | \sD | \tilde\cO_\lambda\big\rangle \coloneqq
\inner{\cO_\lambda | \cO_\lambda}^{-1}$.

\begin{Theorem} \label{thm:instanton-four-point-adjoint-as-horizontal-operator}
 \begin{equation} \label{eq:instanton-four-point-horizontal-operator-formula}
 \sRR = \sRR^H \coloneqq \sD^{(1)} \Gamma_-\big(\sqrt{qt}\big)^{-1} \Gamma_+ (1 )^{-1} \Gamma_- (-A ) \Gamma_+\big({-}1/A\sqrt{qt}\big),
 \end{equation}
 where a superscript $(-)^{(1)}$ means to act on the first tensor
 factor.
\end{Theorem}
\end{MyParagraph}

\begin{MyParagraph}
Explicit formulas like~\eqref{eq:instanton-four-point-horizontal-operator-formula},
particularly the ``vertex operators'' $\Gamma_\pm(z)$, have some
precursors in the literature under the name of {\it
 Awata--Feigin--Shiraishi} or {\it Ding--Iohara--Miki intertwiners}~\cite{Awata2012}. In some sense, the main novelty in our computation
is the introduction of the auxiliary factor of $\sF$ where the
operators $H_{\pm n}$ (in \eqref{eq:gamma-vertex-operators}) from the
vertical subalgebra of $U_{q,t}\big(\hhat{\lie{gl}}_1\big)$ act. The
operators $\Gamma_\pm(z)$ involve a curious interaction of the
vertical and horizontal subalgebras, in contrast to objects living
only in one slope subalgebra
$U_{q,t}\big(\hat{\lie{gl}}_1\big) \subset U_{q,t}\big(\hhat{\lie{gl}}_1\big)$, e.g.,
the vertex operators in the R-matrix~\cite{Negut2015}. Alternatively,
one can view them as single-slope objects evaluated on
$\sF^V \otimes \sF$ instead of $\sF \otimes \sF$, where
$\Phi\colon U_{q,t}\big(\hhat{\lie{gl}}_1\big) \to \End\big(\sF^V\big)$ is the
``vertical'' Fock representation~\cite{Negut2020}.
\end{MyParagraph}

\begin{MyParagraph}
We now proceed with the proof of
Theorem~\ref{thm:instanton-four-point-adjoint-as-horizontal-operator}.
The refined vertex with preferred direction $\nu$ is
\begin{gather} \label{eq:refined-vertex}
 C_{\lambda\mu\nu}(q,t) \coloneqq \left(\frac{t}{q}\right)^{\frac{\|\mu\|^2 + \|\nu\|^2}{2}}\!\! q^{\frac{\kappa(\mu)}{2}} P_{\nu^t}(q,t) \sum_\eta \left(\frac{t}{q}\right)^{\frac{|\eta|+|\lambda|-|\mu|}{2}}\!\! s_{\lambda^t/\eta}\big(q^{-\rho}t^{-\nu}\big) s_{\mu/\eta}\big(t^{-\rho}q^{-\nu^t}\big),\!\!\!
\end{gather}
where the $s_{\lambda/\mu}(\vec x)$ are skew Schur functions,
$q^{-\rho} t^{-\nu}$ means
$\big(q^{1/2} t^{-\nu_1}, q^{3/2} t^{-\nu_2}, q^{5/2} t^{-\nu_3}, \ldots\big)$,
$\|\mu\|^2 \coloneqq \sum_i \mu_i^2$ and
$\kappa(\mu) \coloneqq \|\mu\|^2 - \|\mu^t\|^2$, and
\[ P_{\nu^t}(q,t) \coloneqq q^{\frac{\|\nu\|^2}{2}} \prod_{\square \in \nu} \frac{1}{1 - q^{\ell(\square)+1} t^{a(\square)}}. \]
Setting $\nu^{(0)} \coloneqq \nu^{(r)} \coloneqq \varnothing$, the
desired quantity is
\begin{equation} \label{eq:adjoint-diagram-formula-horizontal}
 Z_r\big(q,t^{-1}; Q, A, \vec B\big) = \sum_{\nu^{(1)}, \ldots, \nu^{(r-1)}} \prod_{i=1}^r B_i^{|\nu^{(i)}|} Z_{\nu^{(i-1)},\nu^{(i)t}}^H\big(q,t^{-1}; Q, A\big),
\end{equation}
where $Z_{\nu,\nu'}^H$ is the so-called {\it four-point diagram}
\begin{align}
 \includepic{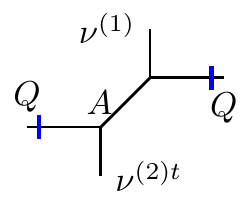}
\eqqcolon{} & Z_{\nu^{(1)}, \nu^{(2)t}}^H\big(q, t^{-1}; Q, A\big) \nonumber \\
={}& \sum_{\lambda,\mu} A^{|\mu|} Q^{|\lambda|} C_{\mu\nu^{(1)}\lambda}\big(q,t^{-1}\big) C_{\mu^t\nu^{(2)t}\lambda^t}\big(t^{-1}, q\big) \nonumber\\
={}& \sum_{\lambda,\mu,\eta_1,\eta_2} Q^{|\lambda|} A^{|\nu^{(2)}|} (qt)^{\frac{\|\lambda^t\|^2 - \|\lambda\|^2}{2}} P_\lambda\big(t^{-1}, q\big) P_{\lambda^t}\big(q, t^{-1}\big) \nonumber \\
 &\qquad\quad \ {}\times s_{\mu^t/\eta_1}\big(q^{-\rho} t^\lambda\big) s_{\nu^{(1)}/\eta_1}\big((qt)^{\frac{1}{2}} q^{-\lambda^t} t^\rho\big) \nonumber \\
 &\qquad\quad \ {}\times s_{\mu/\eta_2}\big(A q^{-\lambda^t} t^\rho\big) s_{\nu^{(2)t}/\eta_2}\big(A^{-1} (qt)^{-\frac{1}{2}} q^{-\rho} t^\lambda\big).\label{eq:four-point-horizontal-partition-function}
\end{align}
The superscript $H$ reminds us that the horizontal direction is
preferred. In the second equality above, we used the homogeneity
$z^{|\lambda|-|\mu|} s_{\lambda/\mu}(\vec x) = s_{\lambda/\mu}(z \vec
x)$ to absorb some factors of $A$ and $\sqrt{qt}$.
\end{MyParagraph}

\begin{MyParagraph}
The following key tool converts a skew Schur function into a matrix
element of an operator on~$\sF$.

\begin{Lemma}[{\cite[Chapter~14]{Kac1990}}] \label{lem:skew-schur-function-as-operator}
 \begin{equation} \label{eq:skew-schur-function-as-operator}
 s_{\lambda/\mu}(\vec x) = \inner*{s_\mu \bigg| \exp\left(\sum_{n > 0} p_n(\vec x) \frac{\alpha_n}{n}\right) \bigg| s_\lambda}.
 \end{equation}
\end{Lemma}

We will need two transformations that can be performed on
\eqref{eq:skew-schur-function-as-operator}, that leave the l.h.s.\
unchanged but modify the r.h.s.:
\begin{itemize}\itemsep=0pt
\item transposing the operator, to get $s_{\lambda/\mu}(\vec x) =
 \inner*{s_\lambda | \exp(\cdots \alpha_{-n} \cdots) | s_\mu}$;
\item applying the $\omega$-involution on $\sF$, to get
 $s_{\lambda/\mu}(\vec x) = \inner*{s_{\mu^t} | \exp(-\cdots p_n(-\vec x)
 \cdots) | s_{\lambda^t}}$.
\end{itemize}
\end{MyParagraph}

\begin{MyParagraph}
For arguments $\vec x = q^{-\rho} t^\nu$ or similar,
Lemma~\ref{lem:skew-schur-function-as-operator} for
$s_{\lambda/\mu}(\vec x)$ is better interpreted as a~matrix element of
an operator on $\sF \otimes \sF$.

\begin{Lemma} \label{lem:skew-schur-function-as-operator-2}
 \begin{gather*}
 s_{\lambda/\mu}(q^{-\rho} t^{\nu}) = \inner*{\tilde\cO_\nu \otimes s_\mu \bigg| \exp\left(-\sum_{n > 0} \big(qt)^{n/2} (t^{n/2} - t^{-n/2}\big) \frac{H_n \otimes \alpha_n}{n}\right) \bigg| \tilde\cO_\nu \otimes s_\lambda}, \\
 s_{\lambda/\mu}(q^{-\nu^t} t^\rho) = \inner*{\tilde\cO_\nu \otimes s_\mu \bigg| \exp\left(\sum_{n > 0} (qt)^{-n/2} \big(q^{n/2} - q^{-n/2}\big) \frac{H_{-n} \otimes \alpha_n}{n}\right) \bigg| \tilde\cO_\nu \otimes s_\lambda}.
 \end{gather*}
\end{Lemma}

\begin{proof}
 For $n > 0$,
 \begin{gather*}
 p_n\big(q^{-\rho} t^{\nu}\big) = q^{n/2} \sum_{k>0} q^{n(k-1)} t^{n\nu_k} = -(qt)^{n/2} \big(t^{n/2} - t^{-n/2}\big) h_n(\nu;q,t),
 \end{gather*}
 where $h_n(\nu; q,t)$ is the eigenvalue of $H_n$ on the fixed point
 $\cO_\nu$, as in~\eqref{eq:tautological-weight-hilb}. Similarly,
 \[ p_n\big(q^{-\nu^t} t^\rho\big) = (qt)^{-n/2} \big(q^{n/2} - q^{-n/2}\big) h_{-n}(\nu^t;t,q), \]
 but quite clearly $h_n(\nu^t;t,q) = h_n(\nu;q,t)$.
\end{proof}
\end{MyParagraph}

\begin{MyParagraph}
It is clear that the terms in the last two lines of~\eqref{eq:four-point-horizontal-partition-function} eventually become
the vertex operators, via Lemma~\ref{lem:skew-schur-function-as-operator-2}, so the terms in the
first line must be absorbed somewhere. Compute that
 \begin{align}
 (qt)^{\frac{\|\lambda^t\|^2 - \|\lambda\|^2}{2}} P_{\lambda^t}\big(q,t^{-1}\big) P_{\lambda}\big(t^{-1},q\big)
 &= q^{\frac{\|\lambda^t\|^2}{2}}t^{-\frac{\|\lambda\|^2}{2}} \prod_{\square \in \lambda} \frac{1}{1 - q^{\ell(\square)+1} t^{-a(\square)}} \frac{1}{1 - q^{\ell(\square)} t^{-a(\square)-1}} \nonumber\\
 &= (-1)^{|\lambda|} (qt)^{\frac{|\lambda|}{2}} \frac{1}{\inner{\cO_\lambda | \cO_\lambda}} \label{eq:four-point-macdonald-factors}
 \end{align}
using that $\sum_{\square \in \lambda} \ell(\square) = \big(\|\lambda^t\|^2
- |\lambda|\big)/2$ and similarly for $a(\square)$. The resulting
$\big({-}\sqrt{qt}\big)^{|\lambda|}$ term is absorbed into the K\"ahler variable
$Q$. Finally, the term $\inner{\cO_\lambda | \cO_\lambda}^{-1}$ comes
from $\sD$.
\end{MyParagraph}

\subsection{Dependence on preferred direction}\label{sec:dependence-on-preferred-direction}

\begin{MyParagraph}
The way in which diagrams such as the ones in
\eqref{eq:adjoint-matter-diagram} depend on the choice of preferred
direction has been raised \cite{Iqbal2009} and studied, e.g.,~\cite{Awata2013}, since the introduction of the refined vertex. A good
way to study this dependence is to relate the refined vertex
$C_{\lambda\mu\nu}(q,t)$ to the more symmetric K-theoretic
Pandharipande--Thomas (PT) vertex $V_{\lambda\mu\nu}(x,y,z;\fq)$.

For a review of (equivariant) K-theoretic DT and PT theory,
\cite{Okounkov2017} should suffice.

\begin{Definition}
 Let $\sT \coloneqq (\bC^\times)^3$ with coordinates $(x, y, z)$.
 Given a function $f(x,y,z)$ on $\sT$ and a cocharacter
 $\sigma(u) = \big(u^a, u^b, u^c\big) \in \sT$, let
 \[ \lim_\sigma f \coloneqq \lim_{u \to 0} f(\sigma(u)). \]
 If $a+b+c = 0$, i.e., $\sigma$ preserves $\kappa \coloneqq xyz$, and
 there is some permutation of $a$, $b$, $c$ so that $a \gg b > 0$, then
 $\sigma$ is called an {\it index limit}.
\end{Definition}

\begin{Theorem}[{\cite[Theorem~2]{Nekrasov2016a}}] \label{thm:PT-to-refined-vertex}
 Assume the DT/PT conjecture {\rm \cite[formula~(16)]{Nekrasov2016a}} for
 equivariant K-theoretic vertices. Then
 \begin{equation*} 
 (\text{\rm prefactor}) \cdot C_{\lambda\mu\nu}\big({-}\fq\kappa^{\frac{1}{2}}, -\fq\kappa^{-\frac{1}{2}}\big) = \lim_{\sigma_V} V_{\lambda\mu\nu}(x, y, z; \fq)
 \end{equation*}
 for the index limit $\sigma_V(u) \coloneqq \big(u^N, u^{-N-1}, u\big)$ with
 $N \gg 0$.
\end{Theorem}

The PT vertex $V_{\lambda\mu\nu}$ is fully symmetric in its three
legs, upon permuting the variables~$x$,~$y$,~$z$ accordingly, so different
permutations of components of the cocharacter $\sigma_V$ produce refined
vertices with different preferred direction.

In general, therefore, refined partition functions $Z$ are index
limits of analogous {\it PT partition functions}
\begin{equation}\label{eq:PT-partition-function}
 Z^{\PT}(x,y,z; \fq, Q,A,\vec B) \in \bQ\big(x^{\frac{1}{2}},y^{\frac{1}{2}},z^{\frac{1}{2}}\big)((\fq))[[Q,A,\vec B]],
\end{equation}
built from the PT vertex $V_{\lambda\mu\nu}$ (and PT edge
contributions) in the same way that $Z$ is built from the refined
vertex $C_{\lambda\mu\nu}$. Changing preferred direction in $Z$
corresponds to changing the index limit $\sigma$ for $Z^{\PT}$, which
we can study geometrically. This sort of approach first appeared in~\cite{Arbesfeld2021} for toric geometries, and we now review (a mild,
non-toric generalization of) the arguments there in order to prove our
Theorem~\ref{thm:adjoint-matter-independence-of-preferred-direction}.
\end{MyParagraph}

\begin{MyParagraph}The new ingredient is the following geometric construction of~$Z_r^{\PT}$, which is no longer associated directly with an actual
toric $3$-fold. We continue to assume the DT/PT conjecture
\cite[formula~(16)]{Nekrasov2016a} throughout this subsection.
\begin{Definition}
 Let $\tilde X$ be the (infinite type) smooth toric $3$-fold given by
 the periodic toric polytope of Figure~\ref{fig:adjoint-geometry},
 where all edges are locally conifolds $\cO(-1)^{\oplus 2}$. Let $\sT
 = (\bC^\times)^3$ be its standard torus. Let $\Lambda_r \cong \bZ^2$
 be the translation action on the polytope with generators as shown
 in the figure, acting on coordinates as
 \begin{equation} \label{eq:translation-action-Xr}
 (x, y, z) \mapsto \big(\kappa x, \kappa^{-1} y, z\big),\qquad
 (x, y, z) \mapsto \big(x, \kappa^{-r} y, \kappa^r z\big)
 \end{equation}
 for $\kappa \coloneqq xyz$.
\end{Definition}

\begin{figure}
 \centering
 \includegraphics{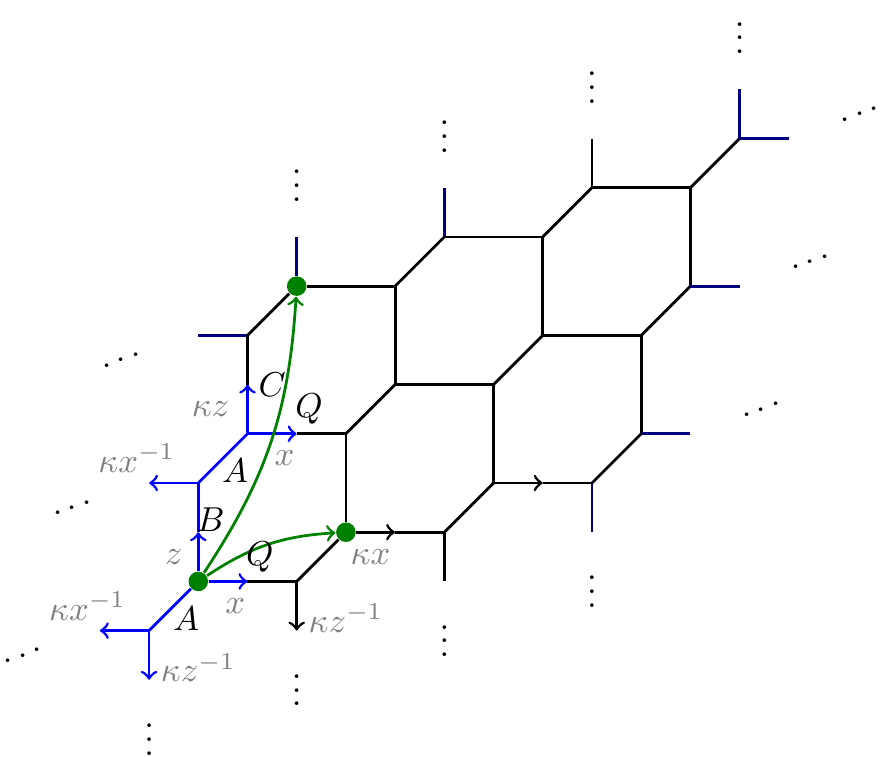}
 \caption{The periodic toric diagram defining $\tilde X$, with
 generators of the translation action $\Lambda_2$ in green and the
 fundamental region in blue. Some coordinates and some K\"ahler
 variables are indicated.}
 \label{fig:adjoint-geometry}
\end{figure}

\begin{Theorem}[\cite{Alexeev2002}]
 There exists a well-defined quotient
 \[ X_r \coloneqq \tilde X/\Lambda_r \to \Spec \bC[[\kappa]] \]
 such that the map to $\Spec \bC[[\kappa]]$ is proper and flat.
\end{Theorem}

This $X_r$ has already appeared explicitly in \cite{Kanazawa2019}, but
it is an instance of a far more general construction of (degenerating)
families of abelian varieties due to Alexeev \cite{Alexeev2002}, where
the initial combinatorial data is an {\it arbitrary} periodic toric
polytope. In general, the generic fiber is an abelian variety while
the special fiber is some union of toric varieties. The prototypical,
rank-$1$ example of Alexeev's construction is the Tate elliptic curve
$k^\times/q^\bZ \to \Spec k[[q]]$.
\end{MyParagraph}

\begin{MyParagraph}
The torus $\sT$ acting on $\tilde X$ can be identified with
$\bC_\kappa^\times \times (\bC^\times)^2$ where $\bC_\kappa^\times$
scales the coordinate~$\kappa$ on the base. On the quotient $X_r$,
localization with respect to $\bC_{\kappa}^\times$ restricts our
attention to the compact special fiber, which is a union of toric
varieties for the remaining $(\bC^\times)^2$. Hence the
$\sT$-equivariant PT theory of $X_r$ is well-defined, and the vertex
formalism is still applicable and yields exactly the desired non-toric
gluings.

\begin{Definition}
 Let $\PT(X_r)$ be the moduli scheme of PT pairs on $X_r$ and
 $\hat\cO^{\vir} \in K_{\sT}(\PT(X_r))$ be its symmetrized virtual
 structure sheaf; see, e.g., \cite[Section~3.2]{Okounkov2017} for details.
 (The symmetrization necessitates passing to a double cover of $\sT$,
 hence the square roots in \eqref{eq:PT-partition-function}.)
 Consider the $\sT$-equivariant K-theoretic pushforward
 \begin{equation}\label{eq:adjoint-PT-function}
 Z_r^{\PT}(x, y, z; \fq; Q, A, \vec B) \coloneqq \chi_{\sT}\big(\PT(X_r), \hat\cO^{\vir} \cdot \fq^{|\cdot|} Q^{\cdots} A^{\cdots} \vec B^{\cdots} C^{\cdots} \big)\big|_{C=0},
 \end{equation}
 where $\fq$ is the boxcounting parameter and the K\"ahler variables
 $Q, A, \vec B, C$ record curve classes as indicated in
 Figure~\ref{fig:adjoint-geometry}, in exactly the same way as for
 the refined partition function $Z_r$. To match with $Z_r$, we set $C
 = 0$ to disallow non-trivial curves (i.e., partitions) along those edges.
\end{Definition}

\begin{Remark}
 The identifications \eqref{eq:kahler-variable-identifications} of
 the K\"ahler variables $Q$ and $\vec A$ come from the relations of
 their corresponding curve classes in $H_2\big(\tilde X, \bZ\big)$, and
 therefore in $H_2(X_r, \bZ)$ as well, by examining the hexagons in
 Figure~\ref{fig:adjoint-geometry}. In contrast, it is {\it not}
 necessary to specialize the $B_i$ in any way.
\end{Remark}
\end{MyParagraph}

\begin{MyParagraph} The proof the following theorem
will occupy the remainder of this subsection.
\begin{Theorem} \label{thm:adjoint-matter-independence-of-preferred-direction}
 With the specializations \eqref{eq:kahler-variable-identifications},
 we have
 \[ Z_r = \lim_\sigma Z_r^{\PT} \]
 which furthermore is independent of the choice of index limit
 $\sigma$. In particular, \eqref{eq:adjoint-matter-diagram} can be
 computed under {\it any} choice of preferred direction.
\end{Theorem}

 Note that the
construction of $Z^{\PT}$ and its independence of $\sigma$ is much
more generally applicable to any toric diagram with appropriate
non-toric gluings, not just our diagram for $Z_r$. In particular there
is no need to set $C=0$ in \eqref{eq:adjoint-PT-function}, in which
case there are obvious so-called ``triality'' symmetries in the
diagram of Figure~\ref{fig:adjoint-geometry}, see, e.g.,~\cite{Bastian2018}.
\end{MyParagraph}

\begin{MyParagraph}
The PT vertex $V_{\lambda\mu\nu}(x,y,z;\fq)$ (and PT edges) enjoys a
number of nice properties, chief among which is that it is a sum of
so-called ``balanced'' rational functions of the form
\begin{equation} \label{eq:balanced-rational-function}
 \prod_i \frac{(\kappa w_i)^{1/2} - (\kappa w_i)^{-1/2}}{w_i^{1/2} - w_i^{-1/2}}, \qquad \kappa \coloneqq xyz,
\end{equation}
for monomials $w_i = w_i(x,y,z)$. (This follows immediately from the
construction of $\hat\cO^{\vir}$.) We refer to the $w_i$ as {\it
 poles} of the rational function.

\begin{Lemma} \label{lem:index-limits-of-balanced-functions}
 Let $\sigma$ and $\tau$ be cocharacters such that
 $\kappa(\sigma(u))$ and $\kappa(\tau(u))$ are independent of~$u$.
 Let~$f$ be a balanced rational function, as in~\eqref{eq:balanced-rational-function}. Then
 \[ \lim_\sigma f = \lim_\tau f \]
 if $\lim_\sigma w_i = \lim_\tau w_i$ for every pole $w_i$ of $f$.
\end{Lemma}

Being built from PT vertices (and edges), $Z_r^{\PT}(x,y,z)$ is also a
sum of balanced rational functions. Hence the behavior of
$\lim_\sigma Z_r^{\PT}$ is controlled by the poles of $Z_r^{\PT}$, and
it suffices to locate these poles.

Note that individual components, e.g., the PT vertices
$V_{\lambda\mu\nu}^{\PT}$ comprising the sum $Z_r^{\PT}$, may have
more poles than $Z_r^{\PT}$ does; there will generally be some pole
cancellation which we will now explain.
\end{MyParagraph}

\begin{MyParagraph}\label{sec:PT-pole-cancellation} Pole-cancellation is controlled by the following geometric observation.
\begin{Proposition}[{\cite[Proposition 3.2]{Arbesfeld2021}}] \label{prop:localization-poles}
 Let~$\cM$ be a space with action by a torus $\sT$. Let~$\cF$ be a~$($virtual$)$ sheaf on $\cM$ and assume that $($virtual$)$ $\sT$-equivariant
 localization is applicable to~$\chi_{\sT}(\cM, \cF)$. Then its poles
 occur only at weights $w \in \Hom(\sT, \bC^\times)$ such that the
 fixed locus~$\cM^{\ker w}$ is non-compact.
\end{Proposition}

\begin{proof}
 If $\cM^{\ker w}$ were compact, then equivariant localization with
 respect to the maximal torus $\sT_w \subset \ker w \subset \sT$
 produces poles only at $\sT$-weights occuring in the (virtual)
 normal bundle~$\cN_{\cM/\cM^{\sT w}}$, none of which vanish on
 $\sT_w$ (by definition of the normal bundle).
\end{proof}

In the case of PT theory of a $3$-fold $X$, the only way for
$\PT(X)^{\ker w}$ to be non-compact is if~$\ker w$ leaves some
non-compact direction in $X$ invariant~-- e.g., if $X$ is toric and
$w$ is an integer power of the weight of some half-edge in the toric
diagram~-- and there are complete curves in~$X$ which can escape to
infinity along that direction. This is a reflection of the fact that
there is a~Hilbert--Chow map
\[ \pi\colon \ \PT(X) \to \Chow(X), \]
which is proper on each component of the Chow variety of
$1$-dimensional cycles on $X$, and so non-compact directions in
$\PT(X)$ must arise from~$\Chow(X)$. We have arrived at the conclusion
of~\cite{Arbesfeld2021}: weights of such non-compact directions in~$X$
form walls in the cocharacter lattice, and~$\lim_\sigma Z^{\PT}$ can
only change when~$\sigma$ crosses a wall.
\end{MyParagraph}

\begin{MyParagraph}\label{sec:non-toric-gluings}
We now move beyond the results of \cite{Arbesfeld2021}, where this
pole-cancellation principle is applied only to toric $3$-folds. Our
$X_r$ is not toric, but the same argument as in
Section~\ref{sec:PT-pole-cancellation} applies. The only non-compact
direction in $X_r$ is along the base $\Spec \bC[[\kappa]]$, and so all
poles of $Z_r^{\PT}$ occur only at integer powers of $\kappa$. By
Lemma~\ref{lem:index-limits-of-balanced-functions}, such poles do not
affect index limits, which by definition leave $\kappa$ constant.
Hence $\lim_\sigma Z_r^{\PT}$ is independent of index limit $\sigma$,
as desired.
\end{MyParagraph}

\begin{MyParagraph}
Finally, it remains to verify that $Z_r = \lim_\sigma Z_r^{\PT}$. The
only non-trivial step is with computing the index limit of PT edge
contributions in $X_r$, for the half-edges which are glued together in
a non-toric way. This is done via the following trivial observation.

\begin{Lemma} \label{lem:independence-of-kappa}
 In the setting of
 Proposition~{\rm \ref{lem:index-limits-of-balanced-functions}}, modifying
 any weight $w_i$ by multiples of $\kappa$ does not affect
 $\lim_\sigma f$.
\end{Lemma}

In particular, this applies to coordinate changes
$(x, y, z) \mapsto \big(\kappa^a x, \kappa^b y, \kappa^c z\big)$ with
$a+b+c = 0$, and the action of $\Lambda_r$ on $\tilde X$ is generated
by substitutions of this form. In other words, $Z_r^{\PT}$ may change
under the substitutions of \eqref{eq:translation-action-Xr}, but
$\lim_\sigma Z_r^{\PT}$ does not. Hence, for each pair of half-edges
glued in a non-toric way, e.g., one in coordinates $(x,y,z)$ and
another in coordinates $(\kappa^a x, \kappa^{-a} y, z)$, we are free
to pick either of the two coordinate charts to write the edge term.

\begin{Example} \label{ex:PT-adjoint-rank-1}
 In our setting, all edges are local conifolds, i.e., having normal
 bundle $\cO(-1)^{\oplus 2}$. Let $E_\lambda$ be the PT edge
 contribution for a curve of class $\lambda$ on such an edge. Then in
 rank $r=1$,
 \begin{gather*}
 \lim_\sigma\left(\includepic{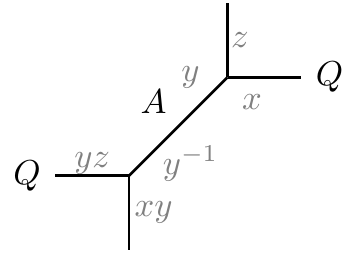}\right)
\eqqcolon \lim_\sigma\big(Z_1^{\PT}(x,y,z; \fq, Q,A,C=0)\big) \\
\qquad= \lim_\sigma\bigg(
  \sum_{\lambda,\mu} Q^{|\lambda|} A^{|\mu|}  V_{\lambda\mu\varnothing}(x,y,z;\fq) V_{\lambda^t\mu^t\varnothing}\big(yz, y^{-1}, xy; \fq\big)
  E_\lambda(x,y,z;\fq) E_\mu(y,z,x;\fq)
\bigg),
 \end{gather*}
 where, equally well, $E_\lambda(x,y,z; \fq)$ could be have been
 replaced by $E_{\lambda^t}\big(yz, y^{-1}, xy; \fq\big)$.
\end{Example}

In general, by Theorem~\ref{thm:PT-to-refined-vertex}, the PT vertex
$V_{\lambda\mu\nu}$ becomes the refined vertex $C_{\lambda\mu\nu}$
(with appropriate preferred direction) up to some prefactors which are
combinatorial quantities in~$\lambda$,~$\mu$, and~$\nu$. One can check
by explicit computation that
\[ \lim_\sigma E_\lambda = q^{\frac{\|\lambda\|^2}{2}} t^{\frac{\|\lambda^t\|^2}{2}} \]
exactly cancels these prefactors, e.g., the part of the prefactor from
the vertex which depends on~$\lambda$ is
$q^{-\|\lambda\|^2/2} (q/t)^{|\lambda|}$, and the vertex on the other
end of the edge contributes
$t^{-\|\lambda^t\|^2/2} (t/q)^{|\lambda^t|}$. See
\cite[Section~4.3]{Arbesfeld2021} for details. We conclude that prefactors
and edges don't matter, and ${\lim_\sigma Z_r^{\PT} = Z_r}$ is just a
combination of refined vertices, as claimed.
\end{MyParagraph}

\subsection{Another explicit operator formula}

\begin{MyParagraph}
From Theorem~\ref{thm:adjoint-matter-independence-of-preferred-direction},
the desired partition function $Z_r\big(q, t^{-1}; Q, A, \vec B\big)$ can also
be computed using four-point diagrams with {\it vertical} preferred
direction:
\begin{align*}
 \includepic{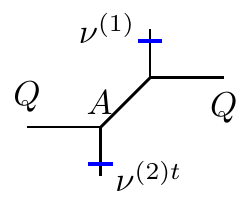}
\eqqcolon{} & Z_{\nu^{(1)}, \nu^{(2)t}}^V\big(q, t^{-1}; Q, A\big) \nonumber \\
={}& \sum_{\lambda,\mu,\eta_1,\eta_2} (QA)^{|\lambda|} \left(qt\right)^{\frac{\|\nu^{(2)t}\|^2 - \|\nu^{(1)}\|^2}{2}} P_{\nu^{(1)t}}\big(q,t^{-1}\big) P_{\nu^{(2)}}\big(q,t^{-1}\big) 
  \\
 &\qquad\quad \ {}\times
 s_{\lambda^t/\eta_1}\big(A^{-1} q^{-\rho}t^{\nu^{(1)}}\big) s_{\mu/\eta_1}\big((qt)^{\frac{1}{2}} A q^{-\nu^{(1)t}}t^\rho\big) \nonumber \\
 &\qquad\quad \ {}\times
 s_{\lambda/\eta_2}\big(q^{-\nu^{(2)t}}t^{\rho}\big) s_{\mu^t/\eta_2}\big((qt)^{-\frac{1}{2}} q^{-\rho} t^{\nu^{(2)}}\big), \nonumber
\end{align*}
cf. \eqref{eq:four-point-horizontal-partition-function}. This is
completely distinct from the horizontal version $Z_{\nu^{(1)},
 \nu^{(2)t}}^H$. Note also that the K\"ahler variable $A$ is
distributed slightly differently this time.
\end{MyParagraph}

\begin{MyParagraph}
We introduce a new ``mixing'' operator
\[ \sM \coloneqq \sum_{\lambda, \mu} \ket{\tilde\cO_\lambda}\bra{\tilde\cO_\mu}. \]

\begin{Theorem} \label{thm:instanton-four-point-adjoint-as-vertical-operator}
 \begin{gather*}
  Z_r\big(q, t^{-1}; QA, A, -\vec B\sqrt{qt}\big) = \tr_{\sF^{\otimes r}} Q^{|\cdot|}
 \big\langle\varnothing \big| \big(\sRR^V\big)^{(01)}\! B_1^{|\cdot|} \big(\sRR^V\big)^{(02)}\! B_2^{|\cdot|} \cdots B_{r-1}^{|\cdot|} \big(\sRR^V\big)^{(0r)}\! \big| \varnothing\big\rangle,
 \end{gather*}
 where the $B_i^{|\cdot|}$ and matrix element are taken in the $0$-th
 tensor factor, and
 \begin{equation} \label{eq:instanton-four-point-vertical-operator-formula}
 \sRR^V = \sD^{(1)} \Gamma_+(1)^{-1} \Gamma_-\big(1/\sqrt{qt}\big)^{-1} \cdot \sM^{(1)} \cdot \Gamma_+\big({-}A\sqrt{qt}\big) \Gamma_- (-1/A ),
 \end{equation}
 where a superscript $(-)^{(1)}$ means to act on the first tensor factor.
\end{Theorem}
\end{MyParagraph}

\begin{MyParagraph}
The proof of
Theorem~\ref{thm:instanton-four-point-adjoint-as-vertical-operator} is
completely analogous to that of
Theorem~\ref{thm:instanton-four-point-adjoint-as-horizontal-operator}
for $\sRR^H$, and so we only provide some comments for the purpose of
comparison.
\begin{itemize}\itemsep=0pt
\item The mixing operator $\sM$ is necessary because for any given
 four-point function, two of the skew Schur functions involve
 $\nu^{(i)}$ while the other two involve $\nu^{(i+1)}$.
\item The computation~\eqref{eq:four-point-macdonald-factors} now
 takes place not within a single four-point function, but across two
 different ones. Explicitly, $\lambda$ in~\eqref{eq:four-point-macdonald-factors} is replaced by $\nu
 \coloneqq \nu^{(i)}$ for a fixed $i$. The resulting
 $\big({-}\sqrt{qt}\big)^{|\nu|}$ term now must be absorbed by the K\"ahler
 variables $\vec B$.
\end{itemize}
\end{MyParagraph}

\begin{MyParagraph}
We make a few comments on the form of
\eqref{eq:instanton-four-point-vertical-operator-formula}. For general
rank $r$,
Theorem~\ref{thm:adjoint-matter-independence-of-preferred-direction}
guarantees that the operators
\[ \big\langle \varnothing \big| \big(\sRR^V\big)^{(01)} \cdots \big(\sRR^V\big)^{(0r)} \big| \varnothing\big\rangle, \; \big\langle\varnothing \big| \big(\sRR^H\big)^{(10)} \cdots \big(\sRR^H\big)^{(r0)} | \varnothing\big\rangle \in \End\big(\sF^{\otimes r}\big) \]
have the same trace, despite being manifestly different operators
(already evident from the lowest-order off-diagonal term). In the case
$r = 1$, the operator $\sM$ does nothing in
$\big\langle\varnothing | \sRR^V | \varnothing\big\rangle \in \End(\sF)$, and one can
check that the traces are equal explicitly. This rank-$1$ calculation
already appeared in \cite[Section~5]{Iqbal2010} in limited generality.
\end{MyParagraph}

\subsection[Properties of RR]{Properties of $\boldsymbol{\sRR}$}
\label{sec:properties-of-RR}

\begin{MyParagraph}
We collect here some properties of R-matrices~-- fusion
(Section~\ref{sec:RR-matrix-fusion}), unitarity
(Section~\ref{sec:RR-matrix-unitary}), and the Yang--Baxter equation
(Section~\ref{sec:RR-matrix-YBE})~-- and discuss their analogues for the
operator $\sRR$. In fact, it is productive to discuss more generally
the fully-equivariant four-point function in K-theoretic PT theory, so
let
\[ \sRR^{\PT}(A; x,y,z) \coloneqq \big(\big(\sRR^{\PT}\big)^{\mu'\nu'}_{\mu\nu}\big) \coloneqq \left(\includepic{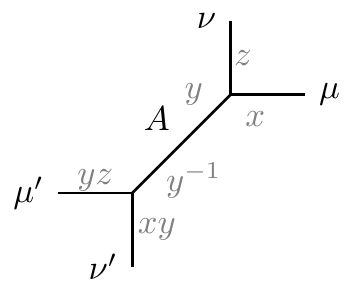}\right) \in \End(\sF \otimes \sF)((A)) \]
denote the PT partition function associated to the four-point diagram
shown. Explicitly, in the notation of
Example~\ref{ex:PT-adjoint-rank-1},
\[ \big(\sRR^{\PT}\big)^{\mu'\nu'}_{\mu\nu} \coloneqq \sum_\lambda A^{|\lambda|} V_{\mu\lambda\nu}(x,y,z;\fq) V_{\mu'\lambda^t\nu'}\big(yz,y^{-1},xy;\fq\big) E_\lambda(y,z,x;\fq). \]
\end{MyParagraph}

\begin{MyParagraph}\label{sec:RR-matrix-fusion}
Recall that if $\sR_{W,V_1} \in \End(W \otimes V_1)$ and $\sR_{W,V_2}
\in \End(W \otimes V_2)$ are R-matrices, then
\[ \sR_{W, V_1 \otimes V_2} = \sR_{W,V_1}^{(12)} \sR_{W,V_2}^{(13)} \in \End(W \otimes V_1 \otimes V_2). \]
In this way, $q$-characters of tensor products arise from R-matrices
of each tensor factor. Completely analogously, the factorization of
\eqref{eq:adjoint-matter-diagram} into four-point diagrams indicates
that we should define
\[ \sRR^{\PT}_{\sF, \sF \otimes \sF} \coloneqq \big(\sRR^{\PT}_{\sF, \sF}\big)^{(12)} B^{|\cdot|} \big(\sRR^{\PT}_{\sF, \sF}\big)^{(13)}, \]
and $qq$-characters of tensor products therefore also arise from
$\sRR$ of each tensor factor. Note that the K\"ahler variable $B$
becomes some combination of evaluation parameters or equivariant
variables under the identification
\eqref{eq:kahler-variable-identifications}.
\end{MyParagraph}

\begin{MyParagraph}\label{sec:RR-matrix-unitary}
Recall that if $\sR(u) \in \End(V \otimes V)$ is a trigonometric
R-matrix with spectral parameter~$u$, then it is important to study
whether
\[ \sR^{(21)}\big(u^{-1}\big) = \sR(u), \]
called {\it unitarity}. We propose that the following is the analogue
for $\sRR^{\PT}$.

\begin{Conjecture}[flop invariance]
 Let
 \[ \tilde\sRR^{\PT}(A; x,y,z) \coloneqq \frac{\sRR^{\PT}(A; x,y,z)}{\sRR^{\PT}(A; x,y,z)^{\varnothing\varnothing}_{\varnothing\varnothing}} \]
 be the normalization, as in \eqref{eq:adjoint-matter-diagram}. Then
 \[ \tilde\sRR^{\PT}\big(A^{-1}; x, xy, z/x\big)^{\mu'\nu}_{\mu\nu'} = \tilde\sRR^{\PT}(A; x,y,z)^{\mu'\nu'}_{\mu\nu}. \]
\end{Conjecture}

The change of variables $(x,y,z) \mapsto (x, xy, z/x)$ is to ensure
that each of the four half-edges~$\mu$,~$\mu'$,~$\nu$,~$\nu'$ retains the
same weight. (Only the weight of the internal edge changes.)

This conjecture is known, by the explicit computation in
\cite{Kononov2021}, when either $\mu = \mu' = \varnothing$ or $\nu =
\nu' = \varnothing$, i.e., only one set of half-edges is non-trivial. In
general it is a question about the behavior of DT (or PT) invariants
under flops, and such general questions have been addressed
non-equivariantly, see, e.g., \cite{Calabrese2016}. In the refined
limit, the conjecture is known in full generality by explicit
computation~\cite{Awata2009}.
\end{MyParagraph}

\begin{MyParagraph}\label{sec:RR-matrix-YBE}
Recall that (trigonometric) R-matrices satisfy the Yang--Baxter
equation
\[ \sR^{(12)}(u) \sR^{(13)}(uv) \sR^{(23)}(v) = \sR^{(23)}(v) \sR^{(13)}(uv) \sR^{(12)}(u), \]
from which one obtains the RTT relation
\[ T^{(13)} T^{(12)} \sR^{(23)} = \sR^{(23)} T^{(13)} T^{(12)}, \qquad T \coloneqq (\sZ \otimes 1)\sR \]
for any operator $\sZ$ such that $[\sZ \otimes \sZ, \sR] = 0$.

\begin{Conjecture}[{\cite[Section~3]{Awata2016}}]
 \[ \sRR^H_{\sF,\sF \otimes \sF} \tilde \sR^{(23)} = \tilde \sR^{(23)} \sRR^H_{\sF,\sF \otimes \sF}, \]
 where
 $\tilde \sR \propto \sR_{\sF,\sF}(1) \in \End(\sF \otimes \sF)$ is a
 certain normalization of the R-matrix.
\end{Conjecture}

From the geometric construction of the R-matrix, one easily obtains
that $\sR \cO_{(\lambda,\mu)} = \cO_{(\mu,\lambda)}$. Hence, as noted
in \cite{Awata2016}, this conjecture reduces to the symmetry
\begin{equation*}
 \big\langle \cO_{\nu'} \otimes \cO_{(\lambda',\mu')} \big| \sRR^H_{\sF,\sF \otimes \sF} \big| \cO_{\nu} \otimes \cO_{(\lambda,\mu)}\big\rangle = \big\langle\cO_{\nu'} \otimes \cO_{(\mu',\lambda')} \big| \sRR^H_{\sF,\sF \otimes \sF} \big| \cO_{\nu} \otimes \cO_{(\mu,\lambda)}\big\rangle,
\end{equation*}
which, to the best of the author's knowledge, is still conjectural.
For example, it is apparently checked in~\cite{Morozov2016} up to
$O\big(Q^3,A^3\big)$.
\end{MyParagraph}

\subsection*{Acknowledgements}

This whole line of inquiry on $qq$-characters was inspired by similar
questions of A.~Okounkov, and benefitted greatly from his and N.
Nekrasov's lectures at the 2019 Skoltech Summer School on Mathematical
Physics. Discussions with N. Arbesfeld, C.-C.M.~Liu, A.~Okounkov, and
R.~Pandharipande were also important, particularly for the contents of
Section~\ref{sec:dependence-on-preferred-direction}. We also the anonymous
referees for bringing \cite{Fukuda2020} to our attention, and for
numerous suggestions that improved the content and exposition in this
paper.
This research was supported by the Simons Collaboration on Special
Holonomy in Geometry, Analysis and Physics.

\pdfbookmark[1]{References}{ref}
\LastPageEnding

\end{document}